\documentclass[a4paper, pdftex]{amsart}

\usepackage[utf8]{inputenc}
\usepackage[T1]{fontenc}
\usepackage[pdftex]{graphicx}
\usepackage{amsmath}
\usepackage{amsfonts}
\usepackage{amsthm}
\usepackage{amssymb}
\usepackage{setspace}
\usepackage[format=hang]{caption}
\usepackage{thmtools}
\usepackage{thm-restate}
\usepackage{mathtools}
\usepackage[pagebackref,colorlinks,citecolor=blue,linkcolor=blue,urlcolor=blue,filecolor=blue]{hyperref}
\usepackage[capitalize]{cleveref}
\usepackage{verbatim}
\usepackage[all, cmtip]{xy}
\usepackage{multirow}
\usepackage{bm}
\usepackage[textsize = footnotesize]{todonotes}

\numberwithin{equation}{section}

\begin{document}
\newtheorem*{thm*}{Theorem}
\newtheorem{theorem}{Theorem}[section]
\newtheorem{corollary}[theorem]{Corollary}
\newtheorem{lemma}[theorem]{Lemma}
\newtheorem{fact}[theorem]{Fact}
\newtheorem*{fact*}{Fact}
\newtheorem{proposition}[theorem]{Proposition}
\newtheorem{claim}[theorem]{Claim}

\newcounter{theoremalph}
\renewcommand{\thetheoremalph}{\Alph{theoremalph}}
\newtheorem{thmAlph}[theoremalph]{Theorem}
\theoremstyle{definition}
\newtheorem{question}[theorem]{Question}
\newtheorem{definition}[theorem]{Definition}

\theoremstyle{remark}
\newtheorem{remark}[theorem]{Remark}
\newtheorem{observation}[theorem]{Observation}
\newtheorem{example}[theorem]{Example}


\newcommand{\benjamin}[1]{\textcolor{red}{[#1]}}
\newcommand{\robin}[1]{\textcolor{blue}{[#1]}}
\newcommand{\yuri}[1]{\textcolor{brown}{[#1]}}

	\renewcommand{\ll}{\left\langle}
	\newcommand{\rr}{\right\rangle}
	\newcommand{\ls}{\left\{}
	\newcommand{\rs}{\right\}}
	\newcommand{\sm}{\setminus}
	\newcommand{\Mf}{\mathfrak}
	\newcommand{\Mbb}{\mathbb}
    \newcommand{\Mbf}{\mathbf}
    \newcommand{\Mcal}{\mathcal}
    \newcommand{\Mrm}{\mathrm}
	
	\newcommand{\mcA}{\ensuremath{\mathcal{A}}}
	\newcommand{\mcB}{\ensuremath{\mathcal{B}}}
	\newcommand{\mcC}{\ensuremath{\mathcal{C}}}
	\newcommand{\mcE}{\ensuremath{\mathcal{E}}}
	\newcommand{\mcF}{\ensuremath{\mathcal{F}}}
	\newcommand{\mcG}{\ensuremath{\mathcal{G}}}
	\newcommand{\mcH}{\ensuremath{\mathcal{H}}} 
	\newcommand{\mcI}{\ensuremath{\mathcal{I}}}
	\newcommand{\mcJ}{\ensuremath{\mathcal{J}}}
	\newcommand{\mcN}{\ensuremath{\mathcal{N}}}
	\newcommand{\mcP}{\ensuremath{\mathcal{P}}}
	\newcommand{\mcQ}{\ensuremath{\mathcal{Q}}}
 	\newcommand{\mcR}{\ensuremath{\mathcal{R}}}
	\newcommand{\mcS}{\ensuremath{\mathcal{S}}}
	\newcommand{\mcU}{\ensuremath{\mathcal{U}}}
	\newcommand{\mcV}{\ensuremath{\mathcal{V}}}
	\newcommand{\mcW}{\ensuremath{\mathcal{W}}}
	\newcommand{\mcZ}{\ensuremath{\mathcal{Z}}}
	\newcommand{\mbR}{\ensuremath{\mathbb{R}}}
	\newcommand{\mbQ}{\ensuremath{\mathbb{Q}}}
	\newcommand{\mbZ}{\ensuremath{\mathbb{Z}}}
    \newcommand{\mbN}{\ensuremath{\mathbb{N}}}
    \newcommand{\mbC}{\ensuremath{\mathbb{C}}}
    \newcommand{\mbbG}{\ensuremath{\mathbb{G}}}

	\newcommand{\on}[1]{\operatorname{#1}}
	\newcommand{\set}[1]{\ensuremath{ \left\lbrace #1 \right\rbrace}}	
	\newcommand{\grep}[2]{\ensuremath{\left\langle #1 \, \middle| \, #2\right\rangle}}
	\newcommand{\real}[1]{\ensuremath{\left\lVert #1\right\rVert}}
	\newcommand{\overbar}[1]{\mkern 2mu\overline{\mkern-2mu#1\mkern-2mu}\mkern 2mu}
	\newcommand{\into}{\hookrightarrow}
	\newcommand{\onto}{\twoheadrightarrow}
	\newcommand{\til}[1]{\widetilde{#1}}
	
	\newcommand{\Aut}{\ensuremath{\operatorname{Aut}}}
	\newcommand{\AutO}{\ensuremath{\operatorname{Aut}^0}}
	\newcommand{\Out}{\ensuremath{\operatorname{Out}}}
	\newcommand{\Outo}[1][A_\Gamma]{\ensuremath{\operatorname{Out}^0(#1)}}
	\newcommand{\Stab}{\operatorname{Stab}}
	\newcommand{\Sym}{\operatorname{Sym}}
	
	\newcommand{\op}{\operatorname{op}}
	\newcommand{\st}{\operatorname{st}}
	\newcommand{\lk}{\operatorname{lk}}
	\newcommand{\im}{\operatorname{im}}
	\newcommand{\rk}{\operatorname{rk}}
	\newcommand{\corank}{\operatorname{crk}}
	\newcommand{\spn}{\operatorname{span}}
	\newcommand{\Ind}{\operatorname{Ind}}

	\newcommand{\GL}[2]{\ensuremath{\operatorname{GL}_{#1}(#2)}}
	\newcommand{\SL}[2]{\ensuremath{\operatorname{SL}_{#1}(#2)}}
	\newcommand{\Sp}[2]{\ensuremath{\operatorname{Sp}_{#1}(#2)}}
	\newcommand{\CD}[2]{\ensuremath{\mathcal{G}_{#1}^{#2}}}
	\newcommand{\uCD}{\ensuremath{\mathcal{G}_\Phi^{{\rm sc}}}}
	\newcommand{\Addi}{\ensuremath{\mathbb{G}_{\mathrm{a}}}}
	\newcommand{\Multi}{\ensuremath{\mathbb{G}_{\mathrm{m}}}}
	
	\newcommand{\field}{\ensuremath{\mathbb{K}}} 
	\newcommand{\ring}{\ensuremath{\mathfrak{O}}} 
	\newcommand{\chevalley}{\ensuremath{\mathcal{G}}} 
	\newcommand{\torus}{\ensuremath{\mathcal{H}}} 
	\newcommand{\St}{\ensuremath{\mathtt{St}}} 
	\newcommand{\vcd}{\operatorname{vcd}} 
	\newcommand{\CC}{\operatorname{CC}} 
	\newcommand{\sign}{\operatorname{sign}} 
	\newcommand{\mfX}{\ensuremath{\mathfrak{X}}} 
	\newcommand{\isom}{f} 
	\newcommand{\isomtwo}{F} 
	
    \newcommand{\tA}{\ensuremath{\mathtt{A}}}
    \newcommand{\tB}{\ensuremath{\mathtt{B}}}
    \newcommand{\tC}{\ensuremath{\mathtt{C}}}
    \newcommand{\tD}{\ensuremath{\mathtt{D}}}
    \newcommand{\tE}{\ensuremath{\mathtt{E}}}
    \newcommand{\tF}{\ensuremath{\mathtt{F}_4}}
    \newcommand{\tG}{\ensuremath{\mathtt{G}_2}}

\title[Top-dimensional cohomology of Chevalley groups]{On the top-dimensional cohomology of arithmetic Chevalley groups}

\author{Benjamin Br\"uck}
\address{Universit\"at M\"unster, Institut f\"ur mathematische Logik und Grundlagenforschung, Einsteinstr.~62, 48149 M\"unster, Germany}
\email{benjamin.brueck@uni-muenster.de}

\author{Yuri Santos Rego} 
\address{University of Lincoln, Charlotte Scott Research Centre for Algebra, School of Mathematics and Physics, Brayford Pool, LN6 7TS Lincoln, United Kingdom}
\email{ysantosrego@lincoln.ac.uk}

\author{Robin J. Sroka}
\address{Department of Mathematics \& Statistics, McMaster University, Canada}
\curraddr{Mathematisches Institut, Universität Münster, Germany}
\email{robinjsroka@uni-muenster.de}

\thanks{BB and RJS were partially supported by the Deutsche Forschungsgemeinschaft (DFG, German Research Foundation) -- Project-ID 427320536 -- SFB 1442, as well as by Germany’s Excellence Strategy EXC 2044 -- 390685587, Mathematics Münster: Dynamics–Geometry–Structure. YSR was partially supported by the German Research Foundation (DFG) through the Priority Program 2026 `Geometry at infinity', Project~62, while working on this project at the OvGU Magdeburg. RJS was partially supported by the European Research Council (ERC grant agreement No.772960) and the Danish National Research Foundation (DNRF92, DNRF151) as a PhD Fellow at the University of Copenhagen, and by NSERC Discovery Grant A4000 in connection with a Postdoctoral Fellowship at McMaster University}

\begin{abstract}
    Let $\field$ be a number field with ring of integers $\ring$ and let $\chevalley$ be a Chevalley group scheme not of type $\tE_8$, $\tF$ or $\tG$. 
    We use the theory of Tits buildings and a result of T\'oth on Steinberg modules to prove that $H^{\on{vcd}}(\chevalley(\ring); \mbQ) = 0$ if $\ring$ is Euclidean.
\end{abstract}

\maketitle

\section{Introduction}

In this article, we obtain the following result about the cohomology of arithmetic groups:

\begin{theorem}
\label{thm_cohomology_vanishing_specific_types}
Let $\field$ be a number field, $\ring$ the ring of integers in $\field$ and $\chevalley$ a Chevalley--Demazure group scheme of type $\tA_n$, $\tB_n$, $\tC_n$, $\tD_n$, $\tE_6$ or $\tE_7$. If $\ring$ is Euclidean, then the rational cohomology of $\chevalley(\ring)$ vanishes in its virtual cohomological dimension $\on{vcd} = \on{vcd}(\chevalley(\ring))$,
\begin{equation*}
	H^{\on{vcd}}(\chevalley(\ring); \mbQ ) = 0.
\end{equation*}
\end{theorem}

As $\on{SL}_n$ and $\on{Sp}_{2n}$ are the simply-connected Chevalley--Demazure schemes of types $\tA_{n-1}$ and $\tC_n$, respectively, \cref{thm_cohomology_vanishing_specific_types} is a common generalisation of results of Lee--Szczarba~\cite{LS76} (for $\SL{n}{\ring}$) and Br\"uck--Patzt--Sroka~\cite[Chapter 5]{Sroka2021} (for $\Sp{2n}{\mbZ}$, building on work of Gunnells \cite{Gun:Symplecticmodularsymbols}).

There are two main ingredients in the proof of  \cref{thm_cohomology_vanishing_specific_types}. 
The first is work of Borel--Serre \cite{BS:Cornersarithmeticgroups} who proved that the groups in question are virtual Bieri--Eckmann duality groups. This allows one to study their high-dimensional rational cohomology by analysing their low-dimensional homology with coefficients in the so called Steinberg module.
The second ingredient is a result of  T\'oth \cite{Toth2005} that gives a generating set of this module. He shows that in the cases covered in \cref{thm_cohomology_vanishing_specific_types}, the Steinberg module is cyclic as a $\chevalley(\ring)$-module. This generalises results by Ash--Rudolph \cite{AR:modularsymbolcontinued} and Gunnells \cite{Gun:Symplecticmodularsymbols} in the cases of $\on{SL}_n$ and $\on{Sp}_{2n}$, respectively.

The Steinberg module can be described as the top-dimensional homology group of an associated Tits building. Previous vanishing results in the settings of $\on{SL}_{n}$ and $\on{Sp}_{2n}$ used explicit descriptions of the buildings that were specific for the corresponding types; see \cite[Sections~1.1 and~4]{Church2019} for type~$\tA$ and \cite[Chapter~5 and Definition~60]{Sroka2021} for type~$\tC$. We prove a building-theoretic generalisation of the key step in \cite[Section 4]{Church2019} for all types; see \cref{lem_inverting_apartment_class}. This enables us to show that cohomology vanishing in the virtual cohomological dimension always follows if one can show that the Steinberg module is generated by ``integral apartment classes''; see \cref{thm_cohomology_vanishing_general}. The generation by integral apartments classes, in turn, is the content of T\'oth's result.

There are two assumptions in  \cref{thm_cohomology_vanishing_specific_types}. The first is that $\chevalley$ be not of type $\tE_8$, $\tF$ or $\tG$. This is due to the same hypothesis in T\'oth's work and comes from the fact that, in these cases, there is no maximal parabolic subgroup whose unipotent radical is abelian \cite[Section~5]{Toth2005}. This makes certain computations harder in these cases \cite[second paragraph after Theorem 2]{Toth2005}. 
The second assumption is that $\ring$ be Euclidean, which is also a restriction in T\'oth's work. 
However, Euclideanity seems to be a natural assumption for a general statement in the style of \cref{thm_cohomology_vanishing_specific_types}. This is among other things indicated by work of Miller--Patzt--Wilson--Yasaki \cite{MPWY:Nonintegralitysome} who obtain non-vanishing results for $\chevalley = \on{SL}_n$ and certain non-Euclidean PIDs $\ring$.
The condition that $\ring$ should at least be a PID is necessary in a strong sense, at least for $\chevalley = \on{SL}_n$ \cite[Theorem D]{Church2019} and $\chevalley = \on{Sp}_{2n}$ \cite[Theorem 1.1]{Brueck2023b}.

In type $\tA$, for the group $\SL{n}{\ring}$, even stronger vanishing results are already known: Church--Putman \cite{CP:codimensiononecohomology} showed that the rational cohomology of this group vanishes also one degree below its virtual cohomological dimension if $\ring = \mbZ$, and Kupers--Miller--Patzt--Wilson \cite{Kupers2022} proved the same result for $\ring$ the Gaussian or Eisenstein integers. Br{\"u}ck--Miller--Patzt--Sroka--Wilson \cite{Brueck2022} extended this to vanishing of the rational cohomology two degrees below the virtual cohomological dimension for  $\ring = \mbZ$.
These results confirm parts of a conjecture by Church--Farb--Putman \cite{CFP:stabilityconjectureunstable} who asked whether it was generally true that
\begin{equation}
\label{eq_CFP}
    H^{\on{vcd}(\SL{n}{\mbZ})-i}(\SL{n}{\mbZ}; \mbQ ) = 0 \text{ if } i < n-1 = \rk(\on{SL}_n).
\end{equation}

In light of \cref{thm_cohomology_vanishing_specific_types}, one is tempted to ask whether vanishing behaviour similar to \cref{eq_CFP} might also occur for other arithmetic Chevalley groups.
\begin{question}
Let $\field$ be a number field, $\ring$ the ring of integers in $\field$ and $\chevalley$ a Chevalley--Demazure group scheme. If $\ring$ is Euclidean, is it true that
\begin{equation*}
	H^{\on{vcd}(\chevalley(\ring))-i}(\chevalley(\ring); \mbQ ) = 0 \text{ for all } i < \rk(\chevalley)?
\end{equation*}
\end{question}

Currently, evidence for such a vanishing pattern is given by \cref{thm_cohomology_vanishing_specific_types}, the above mentioned results in type $\tA$ and work of Br\"uck--Patzt--Sroka \cite{Brueck2023a} in type $\tC$ that shows that $H^{\on{vcd}(\Sp{2n}{\mbZ})-1}(\Sp{2n}{\mbZ};\mbQ)$ is trivial for $n\geq 2$.

{\small{
\subsection*{Acknowledgements} 
We are indebted to Petra~Schwer for helpful discussions, and to Paul~Gunnells for helpful comments and pointing us to~\cite{Toth2005}. We thank Peter Patzt, Jeremy Miller and Jennifer Wilson for comments on earlier versions of this article and Dan Yasaki for comments on computations in low dimensions. RJS would like to thank his PhD advisor Nathalie Wahl for enlightening conversations and helpful feedback about \cite[Chapter 5]{Sroka2021}. 
We thank the anonymous referee for their comments and suggestions.
}}

\section{Background}

\subsection{Coxeter groups and Coxeter complexes}

Given a finite set $S$, consider a symmetric matrix $M = (m_{s,t})_{s,t \in S}$ whose diagonal entries equal one and all other entries are $\infty$ or integers greater than one. A group $W$ with presentation
\[ W = \ll s \in S \mid (st)^{m_{s,t}} = 1 \text{ for all } s, t \text{ with } m_{s,t} < \infty \rr \]
is called a \emph{Coxeter group}. The pair $(W,S)$ is the corresponding \emph{Coxeter system}, and $S$ is the \emph{Coxeter generating set}. The \emph{rank} of the system $(W,S)$ is the cardinality $\vert S\vert$ of the given generating set. We write $\ell(w)$ for the word length of $w\in W$ with respect to the generating set $S$. The system $(W,S)$ is called \emph{spherical} if the underlying Coxeter group $W$ is finite. The reader is referred to standard textbooks, such as \cite{HumphreysLieRep,GeckPfeifferBook,AB:Buildings}, for further background on Coxeter groups. 

\begin{definition}[{\cite[Chapter~3]{AB:Buildings}}]
Let $(W,S)$ be a Coxeter system. Given $J \subsetneq S$, the subgroup $W_J := \ll J \rr \leq W$ is called a (proper) \emph{standard parabolic subgroup} of $W$. 
The \emph{Coxeter complex} of $(W,S)$, denoted by $\Sigma(W,S)$, is the simplicial complex whose vertices are the cosets (in $W$) of its maximal standard parabolic subgroups, and where $g_0 W_{S\setminus \ls s_0\rs}, \ldots, g_{k}W_{S \setminus \ls s_k\rs} $ span a simplex if and only if 
\begin{equation*}
    g_0 W_{S\setminus \ls s_0\rs} \cap \ldots \cap g_{k}W_{S \setminus \ls s_k\rs} \neq \varnothing.
\end{equation*}
Left multiplication on cosets induces an action of $W$ on $\Sigma(W,S)$. 
\end{definition}

\begin{lemma}[{\cite[Section~2.5 and Proposition~1.108]{AB:Buildings}}] \label{lem:FinitenessSphericalCase}
Suppose $(W,S)$ is spherical. Then, $\Sigma(W,S)$ is $W$-equivariantly homeomorphic to the unit $(|S|-1)$-sphere
in Euclidean space $\mbR^{\vert S\vert}$ where any (conjugate of a) Coxeter generator $s \in S$ of $W$ acts on  $\mbR^{\vert S\vert}$ as a linear reflection.
\end{lemma}

The Coxeter complex $\Sigma(W,S)$ has a distinguished simplex $C$ of maximal dimension $(|S|-1)$, which is called the \emph{fundamental chamber} of $\Sigma(W,S)$ and corresponds to the intersection of all maximal standard parabolic subgroups,
\begin{equation*}
    C = \{W_{S\setminus \{s\}} \mid s \in S\}.
\end{equation*}
\cref{lem:FinitenessSphericalCase} justifies the following definition.

\begin{definition} \label{def:coxeter-standardapartmentclass}
Let $(W,S)$ be a spherical Coxeter system.
The \emph{standard apartment class} of $\Sigma(W, S)$ is the generator 
    $$[\Sigma(W,S)] \in \til{H}_{|S|-1}(\Sigma(W,S); \mbZ)\cong \mbZ$$
    with underlying chain
    $$\sum_{w \in W} (-1)^{\ell(w)}w\cdot C \ \in \til{C}_{|S|-1}(\Sigma(W,S); \mbZ).$$
\end{definition}

\begin{corollary} \label{lem:action-reverses-orientation}
In the setting of \cref{lem:FinitenessSphericalCase}, any Coxeter generator $s \in S$ of $W$ acts by $(-1)$ on the standard apartment class of $\Sigma(W, S)$,
$$s \cdot [\Sigma(W, S)] = - [\Sigma(W, S)].$$
\end{corollary}

\subsection{Chevalley groups and their Weyl groups} \label{sec:Chevalley}
We briefly introduce the group schemes that will come up in our work. All of the material presented in this section is standard and can be found in multiple seminal works on the topic, such as~\cite{ChevalleyTohoku,Ree,Kostant,Steinberg}. We shall mostly follow Steinberg's notation~\cite{Steinberg}.

Let $\Phi$ be a (reduced, crystallographic) irreducible root system --- such root systems have been classified and form the seven classical types $\tA_n$, $\tB_n$, $\tC_n$, $\tD_n$, $\tE_6$, $\tE_7$, $\tE_8$, $\tF$ and $\tG$; cf.~\cite{HumphreysLieRep}. Let $\Mf{g}_\Phi$ be the corresponding complex simple Lie algebra and $\Lambda$ the lattice of weights~\cite{HumphreysLieRep} of some complex representation $\Mf{g}_\Phi \to \Mf{gl}(V)$. By the work of Chevalley~\cite{ChevalleyTohoku} (independently by Ree~\cite{Ree}) and Demazure~\cite{SGA3.3}, one can construct from such data $(\Phi, \Lambda)$ a unique representable functor~\cite{Kostant,Waterhouse} sending commutative unital rings to groups. We denote this functor by $\mcG_\Phi^\Lambda$ and call it a \emph{Chevalley--Demazure group scheme of type $\Phi$}. Given a (commutative) ring $R$ (with unity), we call the group of $R$-points $\mcG_\Phi^\Lambda(R)$ a Chevalley--Demazure group, or more briefly a \emph{Chevalley group} (over $R$). 

Denoting by $\Lambda_{\Mrm{sc}}$ (resp. by $\Lambda_{\Mrm{ad}}$) the full lattice of weights (resp. the root lattice $\spn_\mbZ(\Phi)$) of $\Mf{g}_\Phi$, we have $\Lambda_{\Mrm{sc}} \subseteq \Lambda \subseteq \Lambda_{\Mrm{ad}}$. 
These containments are reflected on groups: there exist central isogenies $\mcG_\Phi^{\Lambda_{\Mrm{sc}}} \xrightarrow{p}{\mcG_\Phi^\Lambda} \xrightarrow{q}{\mcG_\Phi^{\Lambda_\Mrm{ad}}}$. (The `largest' scheme on the left is \emph{universal} or \emph{simply-connected}, and the `smallest' on the right is \emph{adjoint}.) 
The following are typical examples of (infinite) families of Chevalley--Demazure groups.
\[\mcG_{\tA_{n}}^{\Lambda_\Mrm{sc}} = \Mrm{SL}_{n+1}, \, \, \, 
\mcG_{\tA_{n}}^{\Lambda_\Mrm{ad}} = \mathbb{P}\Mrm{GL}_{n+1}, \, \, \,
\mcG_{\tB_n}^{\Lambda_\Mrm{sc}} = \Mrm{Spin}_{2n+1}, \, \, \, 
\mcG_{\tC_n}^{\Lambda_\Mrm{sc}} = \Mrm{Sp}_{2n}, \, \, \,
\mcG_{\tD_n}^\Lambda = \Mrm{SO}_{2n}.\]
For the remainder of the paper we usually omit the root system and the lattice of weights and write $\chevalley := \mcG_\Phi^\Lambda$ to simplify notation.

Fixing an arbitrary total ordering on $\Phi$ gives rise to a subset of simple roots $\Delta$. That is, every $\alpha \in \Phi$ is a unique $\mbZ$-linear combination of elements of $\Delta$, in a way that the coefficients are all either positive or negative. In particular, $\Delta$ allows us to define the set $\Phi^+$ of positive roots --- i.e., those roots whose coefficients with respect to $\Delta$ are all positive. Similarly, $\Phi^- = -\Phi^+$. The \emph{rank} of $\Phi$, denoted $\rk(\Phi)$, is the cardinality of $\Delta$ --- it does not depend on the choice of a subset of simple roots. The \emph{rank} of $\chevalley = \mcG_\Phi^\Lambda$ is defined as $\rk(\Phi)$.

A choice of subset of simple roots $\Delta \subset \Phi$ gives rise to a $\mbZ$-subscheme in $\chevalley=\mcG_\Phi^\Lambda$ that is isomorphic to $\Multi^{\rk(\Phi)}$, where $\Multi$ denotes the multiplicative $\mbZ$-group scheme $\Multi(R) \cong (R^\times,\cdot)$, cf.~\cite[Expos\'e~XXII]{SGA3.3}. We denote such a subgroup by $\mcH \leq \chevalley$ and call it the \emph{standard ($\mbZ$-split) maximal torus} of $\chevalley$.

The structure of Chevalley--Demazure groups is very much constrained by certain subgroups determined by roots. Given $\alpha \in \Phi^+$, one constructs an embedding over $\mbZ$ of the group scheme $\left( \begin{smallmatrix} 1 & \ast \\ 0 & 1 \end{smallmatrix}\right)$ into $\chevalley=\mcG_\Phi^\Lambda$ that is normalised by the torus $\mcH \leq \chevalley$. The image of this subscheme is denoted by $\mfX_\alpha \leq \chevalley$. Similarly, the opposite root $-\alpha$ gives a $\mbZ$-embedding $\left( \begin{smallmatrix} 1 & 0 \\ \ast & 1 \end{smallmatrix}\right) \into \chevalley$, whose image is denoted by $\mfX_{-\alpha}$ and is also normalised by $\mcH$. The subschemes $\mfX_\alpha$ and $\mfX_{-\alpha}$ are called \emph{unipotent root subgroups} of $\chevalley$. We remark that both are isomorphic to the additive $\mbZ$-group scheme $\Addi(R)\cong (R,+)$. Turning to the group of $R$-points, a \emph{unipotent root element} attached to $\alpha \in \Phi^+$ is the image $x_\alpha(r) \in \mfX_{\alpha}(R)$ of the unipotent matrix $\left(\begin{smallmatrix} 1 & r \\ 0 & 1 \end{smallmatrix} \right)$ under the above mentioned embedding, where $r \in R$. (Similarly for $-\alpha$ with $\left(\begin{smallmatrix} 1 & 0 \\ r & 1 \end{smallmatrix} \right)$.) 

The \emph{Weyl group} $N(\mcH)/\mcH$ of a Chevalley--Demazure group scheme $\chevalley$ is the quotient of $N(\mcH) \leq \chevalley$ --- i.e., the normaliser of $\mcH$ in $\chevalley$ --- by the torus $\mcH \leq \chevalley$. As the notation suggests, this construction is functorial; cf.~\cite[Expos\'e XXII.3]{SGA3.3}. On the other hand, the root system $\Phi$ also gives rise to a Coxeter group $W_\Phi$~\cite{HumphreysLieRep}, which has the same classical type as $\Phi$. This group acts by reflections on $\mbR^{\rk(\Phi)} = \mbR^{\vert \Delta\vert}$ preserving the set of roots $\Phi \subset \mbR^{\rk(\Phi)}$. (Here we interpret the set of simple roots $\Delta$ as forming a basis for $\mbR^{\rk(\Phi)}$.)
Denote by $S=\{s_\alpha \mid \alpha \in \Delta\}$ the Coxeter generating set of $W_\Phi$. The relationship between these two groups is described in the sequel. Given an arbitrary base ring $R$, define for each $\alpha \in \Phi$  
\begin{equation*}
    w_\alpha := x_\alpha(1) x_{-\alpha}(1)^{-1} x_\alpha(1)  \in N(\mcH)(R) \leq \chevalley(R),
\end{equation*} 
which is the image of $\left( \begin{smallmatrix} 0 & 1 \\ -1 & 0 \end{smallmatrix}\right)$ under the embedding $\ll \left( \begin{smallmatrix} 1 & \ast \\ 0 & 1 \end{smallmatrix}\right), \left( \begin{smallmatrix} 1 & 0 \\ \ast & 1 \end{smallmatrix}\right) \rr \into \ll \mfX_\alpha, \mfX_{-\alpha}\rr \leq \chevalley$. 

\begin{theorem}[{Chevalley~\cite[\textsection{}III]{ChevalleyTohoku},  Demazure~\cite[Expos\'e XXII.3]{SGA3.3}}] \label{thm:Weylgroup} 
Let $R$ be an arbitrary commutative unital ring. 
There is an isomorphism
    \begin{equation*}
        \isom: N(\mcH)(R)/\mcH(R) \to W_\Phi
    \end{equation*}
that maps $w_\alpha \mcH(R)$ to $s_\alpha$ for all $\alpha \in \Delta$.
\end{theorem}
We write $\til{S} = \{ w_\alpha \in \chevalley(R) = \mcG_\Phi^\Lambda(R) \mid \alpha \in \Delta \}$ and denote by $\til{W} = \langle \til{S} \rangle \leq \chevalley(R)$ the subgroup of $\chevalley(R)$ generated by $\til{S}$. The group $\til{W}$ is called the \emph{extended Weyl group} of $\Phi$.

\subsection{The spherical Building of a Chevalley group}
We now specialise to the case where $R=\field$ is a field.
Following Tits~\cite{TitsBN}, every Chevalley group $\chevalley(\field)$ gives rise to a highly-symmetrical simplicial complex on which $\chevalley(\field)$ acts with strong transitivity properties. As before, we let $\mcH$ be an arbitrary, but fixed, maximal split torus of $\chevalley = \mcG_\Phi^\Lambda$. 

\begin{definition}
The \emph{standard Borel subgroup} of $\chevalley(\field)=\mcG_\Phi^\Lambda(\field)$ is the subgroup 
\[\mcB(\field) = \ll \mcH(\field), \mfX_{\alpha}(\field) \mid \alpha \in \Phi^+ \rr \leq \chevalley(\field).\]
The (proper) \emph{standard parabolic subgroups} of $\chevalley(\field)$ are the proper subgroups 
that contain $\mcB(\field)$. The \emph{spherical} (or \emph{Tits}) building of $\chevalley(\field)$, denoted by $\Delta(\chevalley(\field))$, is the simplicial complex whose vertices are the cosets of maximal standard parabolic subgroups of $\chevalley(\field)$, and $g_0 \mcP_0(\field), \ldots, g_{k}\mcP_{k}(\field)$ span a simplex if and only if 
\[g_0 \mcP_0(\field) \cap \ldots \cap g_{k}\mcP_{k}(\field) \neq \varnothing.\]
Left multiplication on cosets induces an action of $\chevalley(\field)$ on $\Delta(\chevalley(\field))$.
\end{definition}

It should be stressed that, due to the central isogenies $\mcG_\Phi^{\Lambda_{\Mrm{sc}}} \xrightarrow{p} \mcG_\Phi^\Lambda \xrightarrow{q} \mcG_\Phi^{\Lambda_\Mrm{ad}}$, the spherical building for $\chevalley(\field)=\mcG_\Phi^\Lambda(\field)$ depends only on the ground field $\field$ and on the root system $\Phi$, but not on the lattice of weights $\Lambda$; cf. \cite[Proposition~5.4]{TitsBN}. That is, $\Delta(\mcG_\Phi^{\Lambda_{\Mrm{sc}}}(\field)) \cong \Delta(\chevalley(\field)) \cong \Delta(\mcG_\Phi^{\Lambda_{\Mrm{ad}}}(\field))$.  Moreover, as the centre of $\chevalley(\field)$ acts trivially on $\Delta(\chevalley(\field))$ by definition, the action of an element $g \in \mcG_\Phi^\Lambda(\field)$ on $\Delta(\chevalley(\field))$ coincides with that of its image $q(g) \in \mcG_\Phi^{\Lambda_{\Mrm{ad}}}(\field)$. 

It is immediate from the above definition that $\Delta(\chevalley(\field))$ contains a canonical maximal simplex $C$, called the \emph{fundamental chamber}, corresponding to the intersection of all maximal standard parabolic subgroups.
Let $\til{W}$ denote the extended Weyl group defined after \cref{thm:Weylgroup}. We call the subcomplex $\Sigma\subset \Delta(\chevalley(\field))$ spanned by $\{ \til{w}C \mid \til{w} \in \til{W} \}$ the \emph{standard apartment} in the Tits building $\Delta(\chevalley(\field))$. Note that the action of $\chevalley(\field)$ on $\Delta(\chevalley(\field))$ restricts to an action of $\til{W}$ on $\Sigma$.

\begin{theorem}[{Tits~\cite[Theorem~5.2]{TitsBN}}] \label{thm:PropertiesBuilding}
Let $\isom: N(\mcH)(\field)/\mcH(\field) \to W_\Phi$ be the isomorphism of \cref{thm:Weylgroup}.
There is an isomorphism of simplicial complexes
$$\isomtwo: \Sigma \to \Sigma(W_\Phi,S)$$
between the standard apartment $\Sigma$ and the Coxeter complex $\Sigma(W_\Phi,S)$ such that for each $\til{w}\in \til{W}$, the action of $\til{w}$ on $\Sigma$ agrees with the action of $\isom(\til{w} \mcH(\field))$ on $ \Sigma(W_\Phi,S)$.
\end{theorem}

Recall that the Coxeter complex $\Sigma(W_\Phi,S)$ is a simplicial $(\vert S\vert-1)$-sphere (see \cref{lem:FinitenessSphericalCase}). \cref{thm:PropertiesBuilding} therefore allows us to associate a unique homology class to the standard apartment $\Sigma$ of $\Delta(\chevalley(\field))$ by pulling back the standard apartment class $[\Sigma(W_\Phi,S)] \in \til{H}_{\rk(\Phi)-1}(\Sigma(W_\Phi,S);\mbZ)$ of the Coxeter complex (see \cref{def:coxeter-standardapartmentclass}). Namely, fixing for each $w \in W_\Phi$ a representative $\til{w}$ for $w$ in $\til{W} \leq \chevalley(\field)$, the \emph{standard apartment class} $[\Sigma]$ in the Tits building is
\begin{equation*}
    [\Sigma] = \isomtwo^{-1}_*[\Sigma(W_\Phi,S)]= \left[\sum_{w \in W_\Phi} (-1)^{\ell(w)}\til{w}\cdot C \right] \in \til{H}_{\rk(\Phi)-1}(\Delta(\chevalley(\field));\mbZ).
\end{equation*}
The Solomon--Tits Theorem says that the $\chevalley(\field)$-translates of this class generate the entire homology of the building:
\begin{theorem}[{Solomon--Tits~\cite{Sol:Steinbergcharacterfinite}}]
\label{thm:SolomonTits} 
For all $i\not = \rk(\Phi)-2$, the reduced homology $\til{H}_{i}(\Delta(\chevalley(\field));\mbZ)$ is trivial. 
The $\mbZ[\chevalley(\field)]$-module $\St := \til{H}_{ \rk(\Phi)-1}(\Delta(\chevalley(\field));\mbZ)$ is generated (as an abelian group) by $\chevalley(\field)$-translates of $[\Sigma]$. 
\end{theorem}

\begin{definition} \label{def:SteinbergApartments}
The $\mbZ[\chevalley(\field)]$-module $\St$ from \cref{thm:SolomonTits} is called the \emph{Steinberg module} for $\chevalley(\field)$. Its generators $g\cdot[\Sigma]$, $g\in \chevalley(\field)$, are called \emph{apartment classes}.
\end{definition}

\subsection{Actions of arithmetic subgroups on spherical buildings}

If $\field$ is a number field with ring of integers $\ring$ and $\chevalley=\mcG_\Phi^\Lambda$, the subgroup $\chevalley(\ring) \leq \chevalley(\field)$ is an example of an \emph{arithmetic subgroup} of $\chevalley(\field)$; cf.~\cite[Chapter~I.3]{Margulis}. Of course, $\chevalley(\ring)$ inherits from $\chevalley(\field)$ a natural action on the spherical building $\Delta(\chevalley(\field))$. Borel--Serre showed that this action reveals a lot of cohomological information about $\chevalley(\ring)$.
In particular, they proved the following.

\begin{theorem}[{Borel--Serre~\cite{BS:Cornersarithmeticgroups}}]
\label{thm_Borel_Serre_duality}
Let $\chevalley$ be a Chevalley--Demazure group scheme, and $\ring$ the ring of integers of an algebraic number field $\field$. Write $r$ for the (real) dimension of the symmetric space associated to $\chevalley(\ring \otimes_{\mbZ} \mbR)$. Then the virtual cohomological dimension of $\chevalley(\ring)$ is  $\vcd(\chevalley(\ring)) = r - \rk(\Phi)$ and, 
for every $i$, there is an isomorphism
\begin{equation*}
	H^{\vcd(\chevalley(\ring))-i}(\chevalley(\ring); \mbQ ) \cong H_i (\chevalley(\ring); \mbQ \otimes_{\mbZ} \St).
\end{equation*}
\end{theorem}
The theorem of Borel--Serre above follows from their construction of the bordification for the symmetric space attached to $\chevalley(\ring \otimes_{\mbZ} \mbR)$, and the fact that extension (resp.~Weil restriction) of scalars preserves parabolics; see \cite[Theorem~8.41 and Section~9]{BS:Cornersarithmeticgroups} and \cite[Section~1.7]{Margulis}.

\section{Vanishing of cohomology}

In this section we prove our main theorem. Recall that the action of $\chevalley(\ring)$ on $\Delta(\chevalley(\field))$ is induced by left multiplication on cosets of parabolics $g\mcP(\field) \subset \chevalley(\field)$.

\begin{definition}
An apartment class 
$[\mcA] \in \St$ is called \emph{integral} if it is a $\chevalley(\ring)$-translate of the standard apartment class $[\Sigma]$. That is, $[\mcA] = \gamma \cdot [\Sigma]$ for some $\gamma \in \chevalley(\ring)$.
\end{definition}

Generalising works of Ash--Rudolph~\cite{AR:modularsymbolcontinued} and Gunnells~\cite{Gun:Symplecticmodularsymbols}, T\'oth established the following. 

\begin{theorem}[{T\'oth~\cite{Toth2005}}]
\label{thm_Toth_integral_generation}
Suppose the Chevalley--Demazure group scheme $\chevalley$ is not of type $\tE_8$, $\tF$ or $\tG$. If the ring of integers $\ring$ is Euclidean, then the Steinberg module $\St$ for $\chevalley(\field)$ is generated  (as an abelian group) by \emph{integral} apartment classes. 
\end{theorem}

Although T\'oth uses the terminology `simple Chevalley groups' (which might be mistaken as adjoint), we remark that the schemes he considers are in fact of the form $\mcG_\Phi^\Lambda$ for any $\Lambda$; cf. \cite[second paragraph of Section~2]{Toth2005}. Moreover, it is straightforward to verify that his results do hold for all $\mcG_\Phi^\Lambda$ (except the excluded types $\tE_8$, $\tF$, $\tG$) since Chevalley's commutator formulae~\cite{ChevalleyTohoku} are valid for all $\mcG_\Phi^\Lambda$ regardless of the lattice $\Lambda$. 

The following is a key ingredient for our proof of \cref{thm_cohomology_vanishing_specific_types}. It can be seen as a building-theoretic generalisation of an argument used by Church--Farb--Putman in the setting of $\on{SL}_n$ \cite[Proof of Theorem C, last paragraph]{Church2019}.

\begin{proposition}
\label{lem_inverting_apartment_class} If the apartment class $[\mcA] \in \St$ is a translation of the standard apartment class by an element of the normaliser of $\chevalley(\ring)$ in $\chevalley(\field)$, then there exists $\gamma \in \chevalley(\ring)$ such that $\gamma \cdot [\mcA] = -[\mcA]$. 
\end{proposition}
\begin{proof}
Let $[\mcA] = \gamma_1 \cdot [\Sigma]$ with $\gamma_1 \in N_{\chevalley(\field)}(\chevalley(\ring))$ and $[\Sigma]$ the standard apartment class be given.
Let $W_\Phi$ be the Coxeter group associated to $\Phi$, and choose a Coxeter generator $s_\alpha \in W_\Phi$ (hence a reflection of the underlying Euclidean space $\mbR^{\rk(\Phi)}$). By \cref{thm:Weylgroup}, we can find a lift of $s_\alpha$ in $\chevalley(\ring)$, i.e., there exists $\gamma_2 \in \chevalley(\ring)$ such that $\gamma_2 \mcH(\ring)$ maps to $s_\alpha$ under the isomorphism $$\isom:  N(\mcH)(\ring)/\mcH(\ring) \to W_\Phi.$$ By \cref{thm:PropertiesBuilding}, we have 
\begin{equation*}
	\gamma_2 \cdot [\Sigma] = \gamma_2 \cdot \isomtwo^{-1}_*([\Sigma(W_\Phi, S)]) = \isomtwo^{-1}_*(s_\alpha \cdot [\Sigma(W_\Phi, S)]).
\end{equation*}
As $s_\alpha$ reverses orientation by \cref{lem:action-reverses-orientation}, we have $s_\alpha \cdot [\Sigma(W_\Phi, S)]=-[\Sigma(W_\Phi, S)]$.

Now define $\gamma \coloneqq \gamma_1 \gamma_2 \gamma_1^{-1}$, which lies in $\chevalley(\ring)$ because $\gamma_1$ normalises it. Putting the above together, it follows that 
\begin{equation*}
	\gamma \cdot [\mcA] = (\gamma_1 \gamma_2 \gamma_1^{-1}) \cdot (\gamma_1 \cdot [\Sigma]) = (\gamma_1 \gamma_2) \cdot [\Sigma] = \gamma_1\cdot(\gamma_2\cdot[\Sigma]) = \gamma_1 \cdot(-[\Sigma]) = -[\mcA],
\end{equation*}
as desired. 
\end{proof}

In what follows we always consider $\mbQ$ to be a trivial $G$-module, where $G$ is either $\chevalley(\ring)$ or $\chevalley(\field)$.

\begin{theorem}
\label{thm_cohomology_vanishing_general}
Let $\field$ be a number field, $\ring$ the ring of integers of $\field$ and $\chevalley$ a Chevalley--Demazure group scheme. If the Steinberg module $\St$ for $\chevalley(\field)$ is generated by integral apartment classes, then 
\begin{equation*}
	H^{\vcd(\chevalley(\ring))}(\chevalley(\ring); \mbQ) = 0.
\end{equation*}
\end{theorem}
\begin{proof}
By \cref{thm_Borel_Serre_duality}, there is an isomorphism
\begin{equation*}
	H^{\on{vcd}(\chevalley(\ring))}(\chevalley(\ring); \mbQ ) \cong H_0 (\chevalley(\ring); \mbQ \otimes_{\mbZ} \St).
\end{equation*}
The right hand side, in turn, is isomorphic to the module of co-invariants of the $\mbQ[\chevalley(\ring)]$-module $\mbQ \otimes_{\mbZ} \St$,
\begin{equation*}
    H_0 (\chevalley(\ring); \mbQ \otimes_{\mbZ} \St) \cong (\mbQ \otimes_{\mbZ} \St)_{\chevalley(\ring)} \cong \mbQ \otimes_{\mbZ[\chevalley(\ring)]} \St;
\end{equation*}
see, for instance, \cite[Chapter~III.1]{BrownCohomology}. Since $\St$ is assumed to be generated by \emph{integral} apartment classes, it therefore suffices to prove that for every integral apartment class $[\mcA] \in \St$ and every $q\in \mbQ$, we have
\begin{equation*}
	q \otimes_{\mbZ \chevalley(\ring)} [\mcA] = 0 \in \mbQ \otimes_{\mbZ [\chevalley(\ring)]} \St.
\end{equation*}
Let $q$ and $[\mcA]$ be given, with $[\mcA]$ integral. By \cref{lem_inverting_apartment_class}, we can find $\gamma\in \chevalley(\ring)$ such that $\gamma [\mcA] = -[\mcA]$. As $\chevalley(\ring)$ acts trivially on $\mbQ$, it follows that
\begin{equation*}
	q \otimes_{\mbZ \chevalley(\ring)} [\mcA] = q \cdot \gamma \otimes_{\mbZ \chevalley(\ring)} [\mcA] = q \otimes_{\mbZ \chevalley(\ring)} \gamma\cdot [\mcA] = -(q \otimes_{\mbZ \chevalley(\ring)} [\mcA]).
\end{equation*}
Hence $q \otimes_{\mbZ \chevalley(\ring)} [\mcA] = 0$ since $\mathrm{char}(\mbQ) = 0$, which concludes the proof.
\end{proof}

\begin{remark}
The proof of \cref{thm_cohomology_vanishing_general} actually shows that if $\St$ is generated by integral apartment classes, then $H_0 (\chevalley(\ring); R \otimes_{\mbZ} \St) = 0$ for all rings $R$ in which $2$ is invertible.
\end{remark}

\begin{proof}[Proof of \cref{thm_cohomology_vanishing_specific_types}]
By Theorem~\ref{thm_Toth_integral_generation}, the Steinberg modules for the Chevalley groups $\chevalley(\field)$ of the given types are generated by integral apartment classes. The claim thus follows from Theorem~\ref{thm_cohomology_vanishing_general}.
\end{proof}

\bibliographystyle{halpha}
\bibliography{mybibliography}

\end{document}